\documentclass[11pt, reqno]{amsart}

\usepackage[margin=28mm]{geometry}

\usepackage[utf8]{inputenc}

\usepackage{amsmath}
\usepackage{amssymb,booktabs}
\usepackage{amsfonts}
\usepackage{graphicx}
\usepackage{amsthm,longtable}
\usepackage{enumerate, mathtools}
\usepackage{lscape}
\usepackage{dsfont, url}
\usepackage{xcolor, colortbl}
\usepackage{mathtools}
\usepackage{cases}

\DeclarePairedDelimiter{\abs}{\lvert}{\rvert}

\usepackage{setspace}

\definecolor{Gray}{gray}{0.975}

\newtheorem{theor}{Theorem}[section]

\newtheorem{df}[theor]{Definition}

\newtheorem{prop}[theor]{Proposition}

\newtheorem{lem}[theor]{Lemma}

\newtheorem{problem}{Problem}

\newtheorem{rem}[theor]{Remark}

\newcommand{\Z}{\mathds{Z}}
\newcommand{\Q}{\mathds{Q}}
\newcommand{\DP}{\mathds{D}\mathds{P}}
\newcommand{\CP}{\mathds{D}\mathds{P}}
\newcommand{\s}{\mathcal{S}}

\newcommand{\de}{\partial}
\newcommand{\dd}{\mathrm{D}}

\newcommand{\D}{\mathds{D}}

\renewcommand{\epsilon}{\varepsilon}

\newcommand{\R}{\mathds{R}}

\newcommand{\N}{\mathds{N}}

\newcommand{\e}{\varepsilon}
\newcommand{\K }{K\"{a}hler }
\newcommand{\KE }{K\"{a}hler-Einstein }
\newcommand{\pk }{para-K\"{a}hler }

\begin{document}
\title[Toric para-K\"{a}hler-Einstein manifolds immersed in space forms]{Toric para-K\"{a}hler-Einstein manifolds immersed in para-K\"{a}hler space forms}

\author{Gianni Manno}
\address{(G. Manno) Dipartimento di Scienze Matematiche ``G. L. Lagrange'', Politecnico di Torino\\Corso Duca degli Abruzzi 24, 10129 Torino}
\email{giovanni.manno@polito.it}

\author{Filippo Salis}
\address{(F. Salis) Dipartimento di Scienze Matematiche ``G. L. Lagrange'', Politecnico di Torino\\Corso Duca degli Abruzzi 24, 10129 Torino }
\email{filippo.salis@polito.it}

\thanks{The first author gratefully acknowledges support by the project ``Finanziamento alla Ricerca'' under the contract numbers 53\_RBA21MANGIO, and by the PRIN project 2022 “Real and Complex Manifolds: Geometry and Holomorphic Dynamics” (code 2022AP8HZ9). 
The second author is supported by the ``Starting Grant'' under the contact number 53\_RSG22SALFIL. Both authors are members of the GNSAGA of the INdAM}


%\author{Gianni Manno, Filippo Salis}

\begin{abstract}
%We give a complete list, for $n\leq 6$, of non-isometric \inv\ \KE manifolds immersed in a finite dimensional complex projective space endowed with the Fubini-Study metric. This solves, in the aforementioned case, a classical and long-staying problem addressed among others by Calabi and Chern.
%
A classical and long-staying problem addressed, among others, by Calabi and Chern, is that to find a complete list of mutually non-isometric \KE manifolds immersed in a finite-dimensional \K  
space form. We address the same problem in the \pk context and, then, we find a list of mutually non-isometric toric \pk manifolds analytically immersed in a finite-dimensional \pk space form.
% under the extra condition of analytic immersions.
\end{abstract}

\maketitle

%\tableofcontents

\section{Introduction}

%\textbf{RICORDARE DI TANTO IN TANTO LE COORDINATE NULLE}

\subsection{Description of the problem and main results}
An \emph{almost para-complex manifold} is a $2n$-dimensional manifold $M$ provided with a field of endomorphisms $\mathcal{T}$ such that $\mathcal{T}^2=1$, having eigenvalues $1$ and $-1$, whose associated eigendistributions are $n$-dimensional. An almost para-complex manifold whose the aforementioned distributions are integrable, is called a \emph{para-complex manifold}.  
A \emph{\pk manifold} is a para-complex manifold endowed with a symplectic form $\omega$ such that $g=\omega\left(\mathcal{T}(\cdot),\,\cdot\,\right)$ is a pseudo-Riemannian metric. A \pk manifold having constant para-holomorphic sectional curvature is said \emph{\pk space form}: if the curvature is zero then it is called \emph{flat} otherwise \emph{non-flat}.

 %a pseudo-Riemannian metric $g$ against which the field of endomorphism $\mathcal{T}$ is parallel, or, equivalently, the $2$-form $\omega=g\left(\mathcal{T}(\cdot),\,\cdot\,\right)$ is symplectic. 

The formal analogy with the K\"ahler geometry makes it possible to state problems, originally formulated in the \K context, also in the \pk case. For instance, a classical  problem in K\"ahler geometry is the characterization of holomorphic and isometric immersions  into \K space forms, i.e., into   K\"ahler  manifolds of constant holomorphic sectional curvature (see \cite{Cal} and for a modern introduction to this subject \cite{loizedda}). This problem can be stated in a unified way, including also the \pk case, as follows:

%\medskip
%\noindent\textbf{General problem.}
%\emph{To determine whether a \K manifold can be \K immersed into a complex space form.}

%\begin{center}
%\textbf{General problem:}
%
%\emph{To determine \textbf{OR TO CLASSIFY \K MANIFOLDS THAT...} whether a (para-)\K manifold can be (para-)\K immersed into a (para-)complex space form.}
%\end{center}

\begin{center}
\textbf{General problem:}

\emph{To classify (para-)\K manifolds that can be (para-)\K immersed into (para-)\K space forms.}
\end{center}

\medskip
\noindent
In the \K case, despite E. Calabi found in \cite{Cal} some criteria that allow, at least from a theoretical viewpoint, to treat  the above problem, a satisfactory classification is far to be obtained as the problem remains too underdetermined.
In fact, even in special cases of great interest, such as \emph{\KE} manifolds (see e.g. \cite{extr,pacific,loizedda, ms2,salis} for more details), a complete classification is still unknown.
More precisely, one can ask to find \KE manifolds that can be \K immersed into \K space forms.
This problem is indeed still open only when the ambient space has positive holomorphic sectional curvature: in this case, the \KE manifolds are called \emph{projectively induced}. Even if one restricts to the class of \emph{toric} projectively induced \KE manifolds, the problem of their characterization is only partially solved, see e.g. \cite{loitoric, ny, ms2}.

%\smallskip
%In the present paper we are going to consider similar problems in the \pk context, in particular the following:
%\begin{problem}\label{genprob2}
%To determine whether a para-\KE manifold can be \pk  immersed (i.e. para-holomorphically and isometrically) into a \pk space form.
%\end{problem}

In the \pk case, as in the \K setting, the above general problem is actually very challenging. The recent paper \cite{MSpK}, where necessary and sufficient conditions for the existence of \pk  immersions in \pk space forms have been found, is a first step in addressing the issue. 
In fact, one of the main difficulties one meets in facing the aforementioned problem, is that in the para-complex context, unlike the complex one, para-holomorphic functions are not, in general, analytic but only $\mathcal{C}^\infty$-smooth. 

Therefore, we are going to focus our attention to the case in which the \pk immersions are analytic and the \pk metrics admit symmetries similar to the toric ones. More precisely, in this paper, we will study the following problem:

\begin{problem}\label{toricprob}
Classify toric para-\KE manifolds 
%$\left(U\subseteq\D^n,\frac{\tau}{2}\de\bar\de \phi\right)$, where $\phi$ is a toric-invariant potential, 
admitting an analytic \pk  immersion into a \pk space form.
\end{problem}
Another disadvantage in the \pk context, unlike the \K one,  is the absence of a notion of Bochner's coordinates, which makes the study of Problem \ref{toricprob} more involved. We overcome such problem by a case by case analysis.

\medskip
Our main results are contained in the following theorems. Theorem \ref{th.main.1} concerns Ricci-flat \pk manifolds, that turn out to be the only para-\KE manifold embeddable into a flat \pk space form, whereas Theorem \ref{th.main.2} concerns the embeddability of toric non-flat \pk manifolds into non-flat  \pk space forms.
\begin{theor}\label{th.main.1}
Flat \pk space forms are the only toric Ricci-flat \pk manifolds that can be analytically \pk immersed into a \pk space form. In particular, they can be \pk immersed only into another flat \pk space form.
%, but they cannot be \pk immersed into a non-flat \pk space form. 
Moreover, toric Ricci-flat \pk manifolds are the only para-\KE manifolds that can be analytically \pk immersed into a flat \pk space form.
\end{theor}
Let $\DP^N$ be projective space constructed on the algebra of para-complex number $\D$ (cf. Section \ref{sec.pk.diastasis}), endowed with the \pk counterpart of the Fubini-Study metric, namely the
\pk metric $g_{pFS}^N$ admitting
$$ \log\left(|Z_0|^2+\mathellipsis+|Z_{N}|^2 \right)$$
as \pk potential, where $(Z_0,\mathellipsis,Z_{N})$ are homogeneous coordinates and $|\cdot|^2$ denotes the para-complex modulus. 
\begin{theor}\label{th.main.2}
%, where $\tau$ is a para-complex number satisfying $\tau^2=1$ and such that $(1,\tau)$ generates the para-complex numbers $\D$ and $\de$, $\bar\de$, $\|\cdot\|$ are meant in the para-complex context.
The toric para-\KE manifolds
$\CP^{n_1}\times\mathellipsis\times\CP^{n_k}$ with the \pk metric 
$$ \frac{K}{h}\ \bigoplus_{i=1}^k \ (n_i+1)\, g_{pFS}^{n_i},$$
where $h$ denotes the greater common divisor between $\{n_1+1,\mathellipsis,n_k+1\}$, can be analytically \pk immersed for any $K\in\Z^+$ into
$\left(\DP^N\, , \, g_{pFS}^N\right)$ with $N$ large enough.

\noindent
If the dimension $\sum_{i=1}^k n_i$ is less or equal to $2$, they are the only ones.
\end{theor}

%\begin{theor}\label{th.main.1}
%Flat \pk space forms are the only Ricci-flat \pk manifolds that can be analytically \pk immersed into a \pk space form. In particular, they can be \pk immersed into another flat \pk space form, but they cannot be \pk immersed into a non-flat \pk space form. Moreover, Ricci-flat \pk manifolds are the only para-\KE manifolds that can be analytically \pk immersed into a flat \pk space form.
%\end{theor}
%\begin{theor}\label{th.main.2}
%Let $\omega_c$
%%$$
%%\omega_c=\frac{4\tau}{c}\de\bar\de\log \|Z\|^2
%%$$
%be the symplectic form of the para-complex space form $\DP^n$ with para-holomorphic sectional curvature equal to $c$.
%%, where $\tau$ is a para-complex number satisfying $\tau^2=1$ and such that $(1,\tau)$ generates the para-complex numbers $\D$ and $\de$, $\bar\de$, $\|\cdot\|$ are meant in the para-complex context.
%Let $\CP^{n_1}\times\mathellipsis\times\CP^{n_k}$ be  endowed with the sympletic form
%$$q\left(c_1 \omega_{c} \oplus \mathellipsis \oplus c_k \omega_{c}\right), $$
%where $k, q\in\Z^+$, $c_i=\frac{1}{G^{k-1}}\prod_{{ j\neq i}}(n_j+1)$ for $i=1,\mathellipsis,k$ and $G=\mathrm{gcd}(n_1+1,\mathellipsis, n_k+1)$, namely  the greatest common divisor between $n_1+1,\mathellipsis, n_k+1$. 
%All these toric para-\KE manifolds can be analytically \pk immersed into a non-flat \pk space form with para-holomorphic sectional curvature equal to $c$. When the dimension $n\leq 2$, they are the only ones.
%\end{theor}

\subsection{Description of the paper}

In Section \ref{basics}, we recall basic notions concerning para-complex geometry by introducing the algebra of para-complex numbers $\D$ and its cartesian product $\D^n$.

\noindent
In Section \ref{sec.pk.diastasis}, we focus our attention to \pk manifolds, in particular, those of Einstein type. We indeed characterize them in terms of suitable Monge-Ampère equations and consider more closely the distinguished class of toric ones. Then, after introducing the diastasis function (a distinguished \pk potential) in the \pk context in a way similar to that performed in \cite{Cal}, we characterize \pk space forms, i.e., \pk manifolds with constant para-holomorphic sectional curvature, in terms of such function.

\noindent
Section \ref{pk immersion of toric} is dedicated to the proof of Theorems \ref{th.main.1} and \ref{th.main.2}. More precisely, in the beginning of the section, we prove a preliminary result telling that, if a toric \pk manifold can be \pk and analytically immersed into a \pk space form, then it admits a  polynomial (or a logarithm of a polynomial) potential (that turns out to be a diastastis' function). This allows to reformulate Problem \ref{toricprob} in terms of existence of particular solutions of a distinct Monge-Ampère equation.

\noindent
In Section \ref{sec.th.1}, based on the aforementioned reformulation, we prove Theorem \ref{th.main.1}.

\noindent
In Section \ref{sec.th.2} we focus our attention to the proof of Theorem \ref{th.main.2}. As a first step, we refine  Problem \ref{toricprob} as we are interested in a particular class of \pk immersions, i.e., those into non-flat \pk space forms. As a second step, we prove the technical Lemma \ref{pt}, that is crucial for proving the second part of Theorem \ref{th.main.2}, which is addressed by separately analyzing the case of one-dimensional manifolds (Section \ref{sec.n1}) and the case of two-dimensional ones (Section \ref{sec.n2}). The validity of the first part of the theorem is shown by means of  explicit computations presented in Section \ref{sec.n3}.

\subsection*{Notation}
A multi-index $I=(i_1,\dots,i_n)$ is an element of $\N^n$ and its length $|I|$ is defined as the number $|I|:=\sum_{k=1}^ni_k$.  If $(x_1,\dots,x_n)$ are local coordinates, we define the
derivative operators $\frac{\de^{|I|}}{\de x^I}$ as follows:
$$
\frac{\de^{|I|}}{\de x^I}:=\frac{\de^{|I|}}{\de x_1^{i_1}\cdots \de x_n^{i_n}}\,.
$$
%We fix a total order on $\N^n$ such that $I_0=0\in\N^n$ and  $|I_i|\leq|I_{i+1}|$ for any $i\in\N$. This obviously induces a total order on the set of the aforementioned derivative operators $\frac{\de^{|I|}}{\de x^I}$. Once a total order on $\N^n$ is fixed, one can construct a bijection  $\iota$ between any subset $\I\subset\N^n$ with finite cardinality and $\{1,\mathellipsis, \#\I\}\subset\N$ preserving the total orders:
%\begin{equation}\label{iota}
%\iota :I\in\I\longrightarrow \iota(I)\in\{1,\mathellipsis, \#\I\}.
%\end{equation}

\section{Basics of para-complex geometry}\label{basics}
\subsection{Para-holomorphic  functions}

The 2-dimensional algebra over $\R$ of \emph{para-complex numbers} $\D$ is generated by $1$ and $\tau$, where 
$$\tau^2=1\,.
$$
In  analogy with the complex numbers, we are going to adopt the notation used in \cite{hl}: each $z\in\D$ can be written as 
$$
z=x+\tau y\,,
$$
and we are going to refer to $x$ and $y$ as the \emph{real} and \emph{imaginary part} of $z$, respectively.
In analogy with the complex numbers, we define the conjugate of $z$
$$ \bar z=x-\tau y$$
and 
$$|z|^2=z\bar z=x^2-y^2.$$
It will be useful to introduce also another coordinate system on $\D$, described as follows. We switch the basis $(1,\tau)$ with $(e,\bar e)$, where
$$
e=\frac{1}{2}\left(1-\tau\right)\,,\qquad \bar e=\frac{1}{2}\left(1+\tau\right)\,,
$$
and we are going to say that $(u,v)$ are the \emph{null-coordinates} of $z$ if  $z=ue+v\bar e$. Note that null-coordinates are uniquely determined once we fix a coordinate system $z=x+\tau y$.

Now we can translate on $\D^n$  what we said about $\D$. In particular, for any $z,w\in\D^n$, we define 
$$\langle z, w\rangle= \sum_{i=1}^n z_i\bar w_i$$
and
$$\|z\|^2:=\sum_{i=1}^n |z_i|^2.$$

\begin{df}
A  function
$$
\begin{array}{rrcl}
F:&U\subseteq \D^n&\longrightarrow&\D\\
& (z_1,\mathellipsis,z_n)&\longmapsto& g(x_1,y_1,\mathellipsis,x_n,y_n)+ \tau\ h(x_1,y_1,\mathellipsis,x_n,y_n),
\end{array}$$ 
where $z_i=x_i+\tau y_i$, is called \emph{para-holomorphic} if and only if $g$ and $h$ are smooth and
\begin{equation*}%\label{dezbar}
\frac{\de F}{\de\bar z_i}:=\frac{1}{2}\left( \frac{\de g}{\de x_i}-\frac{\de h}{\de y_i}\right)+\frac{\tau}{2}\left( \frac{\de h}{\de x_i}-\frac{\de g}{\de y_i}\right)=0
\end{equation*}
for any $1\leq i\leq n$.
\end{df}
The number $n$ stands for the \emph{``para-complex" dimension} of $\D^n$.
In analogy to the complex setting, the differential operator $\frac{\de}{\de z_i}$ is defined by
\begin{equation*}%\label{dez}
\frac{\de F}{\de z_i}:=\overline{\left( \frac{\de \bar F}{\de \bar z_i}\right)}.
\end{equation*}
\begin{rem}\label{ph1}
Let  $F:U\subseteq \D^n\to\D$ be a para-holomorphic function. By considering the null-coordinates 
$$
(\xi,\eta)=(\xi_1,\mathellipsis,\xi_n,\eta_1,\mathellipsis,\eta_n)
$$ 
on $\D^n$, that, as for $n=1$, are uniquely determined once a coordinate system $z_i=x_i+\tau y_i$ is fixed,
%and $(u,v)$, 
$F$ can be written as follows:
%on $\D^n$ \textbf{DIREI MEGLIO... u e v SONO FUNZIONI REALI?} , namely, by writing $F$ as
$$
F(\xi_1e+\eta_1\bar e,\mathellipsis,\xi_n e+\eta_n\bar e)=u(\xi_1,\eta_1,\mathellipsis,\xi_n ,\eta_n)\ e+ v(\xi_1,\eta_1,\mathellipsis,\xi_n ,\eta_n)\ \bar e\,,
$$ 
where $u$ and $v$ are real functions on an open subset of $\R^{2n}$. 
We then straightforwardly get
$$
\frac{\de F}{\de\bar z_i}=\frac{\de u}{\de \eta_i}e+\frac{\de v}{\de \xi_i}\bar e\,.
$$
Therefore, $F$ is para-holomorphic if and only if $u$ is independent of $(\eta_1,\mathellipsis,\eta_n)$  and $v$ is independent of  $(\xi_1,\mathellipsis,\xi_n)$.
\end{rem}
\begin{df}\label{ph2}
A function 
$$
\begin{array}{rccc}
F:&U\subseteq \D^n&\longrightarrow&\D^m\\
& z=(z_1,\mathellipsis,z_n)&\longmapsto&\big(\, f_1(z),\mathellipsis,f_m(z)\,\big)
\end{array}
$$ 
is para-holomorphic if and only if each component $f_i$ is para-holomorphic.
\end{df}

\section{Para-\K and Para-\KE  manifolds}\label{sec.pk.diastasis}

As we said in the introduction, a para-complex manifold is an almost para-complex manifold such that eigendistributions of the almost para-complex structure $\mathcal{T}$ are integrable: this is the same to require the vanishing of the Nijenhuis tensor associated to $\mathcal{T}$. An equivalent definition is the following.
\begin{df}
A smooth manifold $M^n$ of para-complex dimension $n$ is called \emph{para-complex} if it admits an atlas of para-holomorphic coordinates $(z_1,\mathellipsis,z_n)$, such that the transition functions are para-holomorphic.
\end{df}

Below we shall give a slightly different definition of \pk manifold (respect to that given in the introduction) which takes into consideration the concept of \pk potential, that will be crucial for our purposes. We then clarify, in the Remark \ref{rem.pot}, the relationship between \pk potential and \pk metric.

\begin{df}\label{pk}
A \emph{\pk manifold}  of para-complex dimension $n$ is a para-complex manifold $M$ endowed with a symplectic form $\omega$ (called \emph{\pk form}) such that, for any point $p\in M$, there exists an open neighborhood $U\ni p$ and  a smooth function $\phi:U\to \R$ (called \emph{\pk potential}) satisfying
$$
\omega|_{U}= \frac{\tau}{2}\de\bar\de \phi
$$ 
where
\begin{equation*}\label{eq:di.dibarra.phi}
\de\bar\de\phi:=\sum_{i,j=1}^n\frac{\de^2\phi}{\de z_i\de \bar z_j}dz_i\wedge d\bar z_j\,.
\end{equation*}
\end{df}
\begin{rem}\label{rem.pot}
To a given \pk manifold $(M ,\omega)$ of para-complex dimension $n$ and para-complex structure $\mathcal{T}$, it is associated the pseudo-Riemannian  metric 
$g=\omega\left(\mathcal{T}(\cdot),\,\cdot\,\right)$ on $M$. If $\phi:U\subset M\to\R$ is a \pk potential of $\omega$, the restriction of $g$ to $U$ can be obtained as the real part of 
$$
\sum_{i,j=1}^n\frac{\de^2\phi}{\de z_i\de \bar z_j}dz_i\otimes d\bar z_j\,,
$$
i.e., in null coordinates $(\xi,\eta)$, the above pseudo-Riemannian metric reads as
\begin{equation}\label{eq:metric.from.pot}
\sum_{i,j=1}^n\frac{\de^2\phi}{\de \xi_i\de  \eta_j}d\xi_i\otimes d \eta_j\,.
\end{equation}
\end{rem}
\begin{df}\label{pke}
A \pk manifold $(M ,\omega)$ is called \emph{para-\KE} if  the associated pseudo-Riemannian metric $g$ is Einstein, namely, if the Ricci tensor of $g$ is proportional to $g$ via a constant $\lambda \in \R$:
\begin{equation}\label{eq.Ricci}
\mathrm{Ric}(g)=\lambda g\,.
\end{equation}
\end{df}
Let  $\left(U,\frac{\tau}{2}\de\bar\de \phi\right)$ be a para-\KE manifold with Einstein constant equal to $\lambda$. Then, after a straightforward computation (see Section 5 of \cite{amt}  for detailed computations of the Ricci tensor's components), we have that condition \eqref{eq.Ricci} is expressed by
\begin{equation}\label{eq.Ricci2}
\de\bar\de\left(\log \left|\det\left( \frac{\de^2 \phi}{\de z_i\de \bar z_j}\right)\right|+\frac{\lambda}{2}\phi\right)=0\,, \quad 1\leq i,j\leq n\,.
\end{equation}
 Hence, by taking into account  the system of null coordinates $(\xi,\eta)$, equation \eqref{eq.Ricci2} translates into
\begin{equation}\label{invdiast}
\left|\det\left( \frac{\de^2 \phi}{\de \xi_i\de \eta_j}\right)(\xi,\eta)\right|=e^{-\frac{\lambda}{2}\phi(\xi,\eta)+F(\xi)+G(\eta)}\,, \quad 1\leq i,j\leq n\,,
\end{equation}
where $F$ and $G$ are smooth functions.

\subsection{Toric \pk manifolds}

Let  $U\subseteq \D^n$ be an open neighborhood of 0, where it is defined a \pk potential reading as
$$\phi(|z_1|^2,\mathellipsis, |z_n|^2).$$
Hence, if  $\xi=(\xi_1,\mathellipsis,\xi_n)$ and $\eta=(\eta_1,\mathellipsis,\eta_n)$ are  the null-coordinates, we have that 
\begin{equation}\label{inv}
\phi(|z_1|^2,\mathellipsis, |z_n|^2)=\phi(\xi_1\eta_1,\mathellipsis, \xi_n\eta_n).
\end{equation}
A potential of the above form is called \emph{toric}. A \pk manifold with a toric-invariant potential $\phi$ is called a \emph{toric \pk manifold}. In this case, up to rescaling the coordinates $(\xi,\eta)$, we can assume, w.l.o.g., that
\begin{equation}\label{eq.nonso}
\left|\det\left( \frac{\de^2 \phi}{\de \xi_i\de \eta_j}\right)(\xi,0)\right|=\left|\det\left( \frac{\de^2 \phi}{\de \xi_i\de \eta_j}\right)(0,\eta)\right|=1\, , \quad 1\leq i,j\leq n\,.
\end{equation}
On account of \eqref{eq.nonso}, by evaluating \eqref{invdiast} at $\eta=0$ and recalling \eqref{inv},
 we get that $F$ in \eqref{invdiast} is constant and, in particular, it is equal to $\frac{\lambda}{2}\phi(0)-G(0)$. With a similar reasoning, we can also prove that $G$ in \eqref{invdiast} is constant. 
Therefore,  by replacing  $\phi$ with 
%\begin{equation}\label{repl}
%\Phi(\xi,\eta)= \phi(\xi,\eta) -\phi(0,0),
%\end{equation}
\begin{equation}\label{repl}
\Phi(\xi,\eta)= \phi(\xi_1\eta_1,\mathellipsis, \xi_n\eta_n) -\phi(0),
\end{equation}
so that $\Phi(0,0)=0$,
equation \eqref{invdiast} can be put in the following form:
\begin{equation}\label{MAlambda}
\left|\det\left( \frac{\de^2 \Phi}{\de \xi_i\de \eta_j}\right)\right|=
e^{-\frac{\lambda}{2}\Phi}\,, \quad 1\leq i,j\leq n\,;
\end{equation}

\begin{rem}
We notice that, by introducing the new set of variables
$$y_i=\log\abs{\xi_i\eta_i}$$
and the unknown function
$$\tilde \Phi=\Phi-\frac{2}{\lambda}\sum_{i=1}^n y_i,$$
%By taking into account account that
%$$ \frac{\de^2 \Phi}{\de \xi_i\de \eta_j}=\frac{\de}{\de \xi_i}\left(\frac{1}{\eta_j}\frac{\de \Phi}{\de y_j}\right)=\frac{1}{\xi_i\eta_j}\frac{\de^2 \Phi}{\de y_i\de y_j},$$ 
 the  Monge-Ampère equation  \eqref{MAlambda} reads as 
\begin{equation*}
\left|\det\left( \frac{\de^2 \tilde\Phi}{\de y_i\de y_j}\right)\right|=
e^{-\frac{\lambda}{2}\tilde\Phi}\,,\quad 1\leq i,j\leq n\,.
\end{equation*}
Such  extensively studied equation could provide inputs for new strategies to tackle Problem \ref{toricprob}, as in the past for the \K counterpart (see e.g. \cite{bb}).
\end{rem}

\subsection{Diastasis}\label{sec.diastasis}

Let  $U$ be an open subset of a \pk manifold  $(M,\omega)$ of \pk dimension $n$,  where it is defined a local potential $\Phi$. We  assume that $U$ can be covered by a system of null-coordinates 
$$
(\xi,\eta)=(\xi_1,\mathellipsis,\xi_n,\eta_1,\mathellipsis,\eta_n).
$$
Up to shrinking $U$, we can also assume that it splits as a product 
$$U=\Omega\times \Omega.$$  
According to \cite{MSpK}, we define the \emph{diastasis function}
$$\dd:U\times U=\Omega\times\Omega\times\Omega\times\Omega\to\R$$
as
$$
\dd\left(\xi,\eta,\zeta,\lambda\right)=\Phi(\xi,\eta)-\Phi(\zeta,\eta)-\Phi(\xi,\lambda)+\Phi(\zeta,\lambda).
$$

\begin{prop}[\cite{MSpK}]
The diastasis is a function defined in a neighborhood of the diagonal of the product
manifold $M\times M$. In particular, it is independent of the choice of the \pk potential.
\end{prop}
We define
$$
\dd_p(q):=\dd(q,p)\,.
$$
In particular, taking into account \eqref{repl}, we get
\begin{equation}\label{diast}
\dd_0(\xi,\eta)=\dd(\xi,\eta,0,0)=  \phi(\xi,\eta)-\phi(0,0)=\Phi(\xi,\eta).
\end{equation}
%Notice that we took into account \eqref{inv} and \eqref{repl}.

%For our purposes it will be useful the following proposition.
%
%\textbf{IN REALTA' MI SEMBRA CHE NON LA CITIAMO MAI...}
%
%%
%\begin{prop}[Hereditary Property \cite{MSpK}]\label{her}
%Let $(S,\theta)$  and $(M,\omega)$ be two \pk manifolds and let 
%$$f: (S,\theta)\to(M,\omega)$$ be a \pk immersion. Let $\dd^S$ and $\dd^M$ be  the diastasis functions of $S$ and $M$, respectively. 
%If $p\in S$  and if $\dd^S_p$  is defined on an open subset $U\ni p$ of $S$, then
%$$\dd^S_p(q)=\dd^M_{f(p)}\left(f(q)\right),\qquad \forall q\in U.$$
%\end{prop}

\subsection{Para-\K space forms}
In complete analogy with the case of K\"ahler manifolds, one can define the para-holomorphic sectional curvature (see for instance \cite{gm}). 
%and, so, the \pk space forms as \pk manifold having constant para-holomorphic sectional curvature.
\begin{df}
A \emph{\pk space form} is a \pk manifold with constant para-holomorphic sectional curvature. 
We denote by $\s_c^N$ an $N$-dimensional simply connected para-complex manifold that can be endowed  with a \pk form $\omega_c$ whose associated pseudo-Riemannian metric $g_c$ is complete and has constant para-holomorphic sectional curvature equal to $c$.
\end{df}

\begin{prop}[\cite{gm} Prop. 3.11]
Two complete and simply connected \pk space forms with the same para-holomorphic sectional curvature are para-holomorphically isometric. 
%Furthermore, any complete \pk space form is the quotient of a complete simply connected \pk space form  by a discrete group of para-holomorphic isometries acting properly discontinuously.
\end{prop}

\noindent By \cite{gm}, we have that  an open subset of a \pk space form is para-holomorphically isometric to an open subset of one  of the subsequent models, according to their (constant) para-holomorphic sectional curvature. Therefore, since we are mainly interested in local \pk immersions into open subsets of \pk space forms, we are going to assume that  $\left(\s_c^N,\omega_c\right)$  is  one of the following models.

\begin{description}
\item[Model for the flat case] The model of the flat \pk space form is
\begin{equation*}%\label{omega0}
\left(\s_0^N,\omega_0\right)=\left(\D^N,{\tau}\de\bar\de \|z\|^2\right),
\end{equation*}
whose potential reads in null-coordinates $(\xi_1,\mathellipsis, \xi_N,\eta_1,\mathellipsis,\eta_N)$ as 
\begin{equation*}\label{dn}
 4\sum_{i=1}^N \xi_i\eta_i =\dd_0(\xi,\eta), 
\end{equation*}
where $\dd_0$ is defined by \eqref{diast}, leading, via formula \eqref{eq:metric.from.pot}, to the metric
$$ 4 \sum_{i,j=1}^N d\xi_i\otimes d \eta_j\,.$$

%\noindent
%The space \eqref{omega0}
%is an example of homogeneous space with respect to its para-holomorphic isometry group. Such group consists of translations and $\D$-unitary transformations $z\in\D^N\to Az\in \D^N$, where $A$ belongs to the \emph{$\D$-unitary group} \textbf{TOGLIERE?}
%$$\mathrm{U}_N(\D)=\{A\in \D^{N,N}\ |\ \|Aw\|^2=\|w\|^2\ \forall w\in\D^N\}.$$

\smallskip
\item[Models for the non-flat cases]
Similarly to the real and complex setting, the para-complex projective space $\DP^N$ can be defined as the quotient of
$$\{Z\in\D^{N+1} \ |\ \|Z\|^2>0\}$$
under the equivalence relation given by
$Z\sim W$ if and only if there exists $\alpha\in\D$ such that $Z=\alpha W$ with $|\alpha|^2>0$. Then, our model of non-flat \pk space form will be
\begin{equation*}
(\s_c^N,\omega_c)=\left(\DP^N,\frac{4\tau}{c}\de\bar\de\log \|Z\|^2 \right).
\end{equation*}
In null-coordinates  $(\xi_1,\mathellipsis, \xi_N,\eta_1,\mathellipsis,\eta_N)$ of the affine chart $\mathcal{U}_\alpha:=\{[Z_0,\mathellipsis,Z_N]\in \DP^N \ |\ |Z_\alpha|^2\neq 0\}$, where $\alpha=1,\mathellipsis,n$, i.e., $\xi_i e+\eta_i \bar e= \frac{Z_i}{Z_\alpha}$ for any $i\neq \alpha$, the potential is equal to 
\begin{equation}\label{dpn}
\frac{8}{c}\log\left( 1+2\sum_{i=1}^N \xi_i\eta_i\right)=\dd_0(\xi,\eta),
\end{equation}
where $\dd_0$ is defined by \eqref{diast}, leading, via formula \eqref{eq:metric.from.pot}, to  the metric $\frac{8}{c} g_{pFS}^N$, with $g_{pFS}^N$ the para-Fubini-Study metric defined after Theorem \ref{th.main.1}.
%Since the action of $\mathrm{U}_{N+1}(\D)$ passes to the quotient, we can easily see that these \pk space forms are homogeneous with respect to the action of their para-holomorphic isometry groups.
\end{description}

\section{Para-\K immersions of toric para-\KE manifolds in \pk space forms}{\label{pk immersion of toric}}

According to Section \ref{sec.diastasis}, let $U=\Omega\times\Omega$ be an open subset of a \pk manifold of para-complex dimension equal to $n$.
As shown in the Section 3 of \cite{MSpK}, when $(U,\frac{\tau}{2}\de\bar\de \phi)$ admits a \pk immersion into the \pk space form $\s_c$, the diastasis function $\dd_0$ reads as
%\begin{equation}\label{dd0}
%\dd_0(\xi,\eta)=
%\begin{cases}
%4\sum_{i=1}^N u_i(\xi)v_i(\eta) &\text{if } c=0;\\
%\frac{8}{c}\log\left(1+2\sum_{i=1}^N u_i(\xi)v_i(\eta)\right) &\text{if } c\neq 0,
%\end{cases}
%\end{equation}
\begin{subnumcases}{\dd_0(\xi,\eta)=}
  4\sum_{i=1}^N u_i(\xi)v_i(\eta) & \text{if} $c=0$; \label{dd0a}
   \\
   \frac{8}{c}\log\left(1+2\sum_{i=1}^N u_i(\xi)v_i(\eta)\right) & \text{if} $c\neq0$\,, \label{dd0b}
\end{subnumcases}
where $u_i:\Omega \to\R$ and $v_i:\Omega\to\R$ are suitable smooth functions.

As we said in the introduction, one of the main differences one meets in the para-complex context, unlike the complex one, is that para-holomorphic functions are not, in general, analytic, but only $\mathcal{C}^\infty$-smooth. This makes the para-complex geometry much less rigid with respect to the complex one. Taking this into account, we restrict our attention to the case of analytic immersion and toric-invariant \pk potential, hence the study of Problem \ref{toricprob}. In fact, as we shall see below, in this case the analyticity condition implies polynomiality. More precisely we have the following lemma.

\begin{lem}\label{lem.polyn}
Let $\phi$ be a toric-invariant \pk potential.
If $u_i$ and $v_i$ are some analytic functions such that
\begin{equation}\label{eq:1}
\sum_{i=1}^N u_i(\xi)v_i(\eta)=\phi(\xi_1\eta_1,\mathellipsis,\xi_n\eta_n),\end{equation}
where  $N\in\N$ is the smallest possible,
then they are polynomials. Hence, $\phi$ is a polynomial in the variables 
\begin{equation}\label{eq.x.xi.eta}
x_i=\xi_i\eta_i\,.
\end{equation}
\end{lem}
\begin{proof}
By differentiating \eqref{eq:1} and evaluating at $\xi=0$, we obtain the following (compatible) linear system
\begin{equation}\label{system}
\begin{pmatrix}
\frac{\de^{|I_1|} u_1}{\de\xi^{I_1}}(0) &\dots&\frac{\de^{|I_1|}  u_N}{\de\xi^{I_1}}(0) \\
\vdots& &\vdots\\
\frac{\de^{|I_N|}  u_1}{\de\xi^{I_N}}(0) &\dots&\frac{\de^{|I_N|}  u_N}{\de\xi^{I_N}}(0) \\
\end{pmatrix}
\begin{pmatrix}
v_1(\eta)\\
\vdots\\
v_N(\eta)
\end{pmatrix}
=\begin{pmatrix}
\eta^{I_1}\frac{\de^{|I_1|}  \phi}{\de x^{I_1}}(0)\\
\vdots\\
\eta^{I_N}\frac{\de^{|I_N|}  \phi}{\de x^{I_N}}(0)
\end{pmatrix},
\end{equation}
where $I_i\in\N^n$ are arbitrary chosen multi-indices. As a first step, we notice that there exist some $I_1,\mathellipsis,I_N\in\N^n$ such that
$$\mathrm{rank} \begin{pmatrix}
\frac{\de^{|I_1|} u_1}{\de\xi^{I_1}}(0) &\dots&\frac{\de^{|I_1|} u_N}{\de\xi^{I_1}}(0) \\
\vdots& &\vdots\\
\frac{\de^{|I_N|} u_1}{\de\xi^{I_N}}(0) &\dots&\frac{\de^{|I_N|} u_N}{\de\xi^{I_N}}(0) \\
\end{pmatrix}>0.
$$ 
Indeed, if, on the contrary, the previous matrix had rank equal to 0 for any $I_1,\mathellipsis,I_N$, then, in view of the analyticity of $u_1,\dots,u_N$, we would have
$$u_1=\mathellipsis=u_N\equiv 0$$  and $\phi\equiv0$ would  not be a potential of a symplectic form.

\smallskip\noindent
Now, if
$$\mathrm{rank} \begin{pmatrix}
\frac{\de^{|I_1|}  u_1}{\de\xi^{I_1}}(0) &\dots&\frac{\de^{|I_1|}  u_N}{\de\xi^{I_1}}(0) \\
\vdots& &\vdots\\
\frac{\de^{|I_N|}  u_1}{\de\xi^{I_N}}(0) &\dots&\frac{\de^{|I_N|}  u_N}{\de\xi^{I_N}}(0) \\
\end{pmatrix}=N$$
for some $I_1,\mathellipsis,I_N\in\N^n$, then, being $(v_1,\mathellipsis,v_N)$ a solution of the linear system \eqref{system}, $v_1,\mathellipsis,v_N$ are polynomials.

\noindent
To conclude, we take into account the case in which
$$\mathrm{rank} \begin{pmatrix}
\frac{\de^{|I_1|} u_1}{\de\xi^{I_1}}(0) &\dots&\frac{\de^{|I_1|} u_N}{\de\xi^{I_1}}(0) \\
\vdots& &\vdots\\
\frac{\de^{|I_N|} u_1}{\de\xi^{I_N}}(0) &\dots&\frac{\de^{|I_N|} u_N}{\de\xi^{I_N}}(0) \\
\end{pmatrix}<N$$
for any $I_1,\mathellipsis,I_N\in\N^n$. In view of what we said above, we can assume that
\begin{equation}\label{rank1}
\mathrm{rank} \begin{pmatrix}
\frac{\de^{|I_1|} u_1}{\de\xi^{I_1}}(0) &\dots&\frac{\de^{|I_1|} u_\rho}{\de\xi^{I_1}}(0) \\
\vdots& &\vdots\\
\frac{\de^{|I_\rho|} u_1}{\de\xi^{I_\rho}}(0) &\dots&\frac{\de^{|I_\rho|} u_\rho}{\de\xi^{I_\rho}}(0) \\
\end{pmatrix}=\rho>0\end{equation}
for some suitable  $I_1,\mathellipsis,I_\rho\in\N^n$ and
\begin{equation}\label{rank2}
\mathrm{rank} \begin{pmatrix}
\frac{\de^{|I_1|} u_1}{\de\xi^{I_1}}(0) &\dots&\frac{\de^{|I_1|}  u_{\rho+1}}{\de\xi^{I_1}}(0) \\
\vdots& &\vdots\\
\frac{\de^{|I_\rho|}  u_1}{\de\xi^{I_\rho}}(0) &\dots&\frac{\de^{|I_\rho|} u_{\rho+1}}{\de\xi^{I_\rho}}(0) \\
\frac{\de^{|J|} u_1}{\de\xi^{J}}(0) &\dots&\frac{\de^{|J|} u_{\rho+1}}{\de\xi^{J}}(0) \\
\end{pmatrix}=\rho\end{equation}
for any $J\in\N^n$.
It follows from the Laplace's expansion w.r.t. the last row of the matrix in \eqref{rank2}, and by considering also \eqref{rank1}, that
\begin{equation*}%\label{taylor}
\frac{\de^{|J|} u_{\rho+1}}{\de\xi^{J}}(0)=-\frac{K_1}{K_{\rho+1}} \frac{\de^{|J|} u_1}{\de\xi^{J}}(0)-\mathellipsis -\frac{K_\rho}{K_{\rho+1}} \frac{\de^{|J|} u_\rho}{\de\xi^{J}}(0),\qquad \forall J\in\N^n,
\end{equation*}
where $K_i$ denotes the algebraic complement of $i$-th element of the last row of the matrix in \eqref{rank2}.
Since
$$
u_{\rho+1}(\xi)=\sum_{J\in\N^n}\frac{\de^{|J|} u_{\rho+1}}{\de\xi^{J}}(0)\  \frac{\xi^J}{J!} =-\frac{K_1}{K_{\rho+1}} u_1(\xi)-\mathellipsis-\frac{K_\rho}{K_{\rho+1}} u_\rho(\xi)\,,
$$
by taking suitable linear combinations $\tilde v_1,\mathellipsis,\tilde v_\rho$ of $v_1,\mathellipsis,v_n$, we have 
$$\sum_{i=1}^\rho u_i\tilde v_i=\phi.$$ Hence, $N$ in \eqref{eq:1} would not be the smallest integer, as assumed.

\smallskip\noindent
The same reasoning can be applied to the functions $u_i$.
\end{proof}
In view of Lemma \ref{lem.polyn} and by taking into account \eqref{dd0a}-\eqref{dd0b}, we straightforwardly get the following proposition.
\begin{prop}
Let $\phi$ be a toric-invariant \pk potential defined on a open neighborhood $U\subseteq \D^n$ of the origin. Let $(U,\frac{\tau}{2}\de\bar\de \phi)$ be a \pk manifold admitting an analytic \pk immersion into the \pk space form $\s_c$. Recalling that $x_i=\xi_i\eta_i$ (cf. \eqref{eq.x.xi.eta}), the diastasis function $\dd_0:U\to\R$ reads as
%\begin{equation}\label{diastind}
%\dd_0(x_1,\mathellipsis, x_n)=
%\begin{cases}
%\sum_{\substack{I\in\N^n,\ 0<|I|\leq d}} a_I x^I&\text{if } c=0;\\
%\frac{8}{c}\log\left(1+\sum_{\substack{I\in\N^n,\ 0<|I|\leq d}} a_I x^I\right) &\text{if } c\neq 0,\\
%\end{cases}
%\end{equation}
%where $d\in\Z^+$ and $a_I\in\R$.
\begin{subnumcases}{\dd_0(x_1,\mathellipsis, x_n)=}
 \sum_{\substack{0<|I|\leq d}} a_I x^I & \text{if} $c=0$; \label{diastind_a}
   \\
   \frac{8}{c}\log\left(1+\sum_{\substack{0<|I|\leq d}} a_I x^I\right) & \text{if} $c\neq0\,,$ \label{diastind_b}
\end{subnumcases}
where $d\in\Z^+$, $a_I\in\R$ and $I\in\N^n$.
\end{prop}
To sum up, by taking into account  \eqref{MAlambda} and \eqref{diast}, we have that Problem \ref{toricprob} can be reformulated as follows:

\medskip\noindent
\textbf{\large Reformulation of Problem \ref{toricprob}.}
\emph{Find all the solutions of type \eqref{diastind_a}-\eqref{diastind_b} of the following Monge-Ampère equation:}
\begin{equation}\label{lambda1}
\left|\det\left( \frac{\de^2 \dd_0}{\de x_i\de x_j}x_i+\frac{\de \dd_0}{\de x_j}\delta_{ij}\right)\right|=
 e^{-\frac{\lambda}{2}\dd_0}\,, \quad 1\leq i,j\leq n\,.
\end{equation}

\medskip

\subsection{Proof of Theorem \ref{th.main.1}}\label{sec.th.1}

In view of the above Reformulation of Problem \ref{toricprob}, in order to get Theorem \ref{th.main.1}, it will be enough to prove the following properties:
\begin{itemize}
\item there are no solutions of type \eqref{diastind_a} of the Monge-Ampère equation \eqref{lambda1} with $\lambda\neq 0$;
\item there are no solutions of type \eqref{diastind_b} of the Monge-Ampère equation \eqref{lambda1} with $\lambda=0$;
\item  $P$ is a solution of type \eqref{diastind_a} of the Monge-Ampère equation \eqref{lambda1} with $\lambda=0$,  if and only if $\deg P=1$ and the product of the coefficients of the linear terms is $\pm 1$.
\end{itemize}

Let $x=(x_1,\mathellipsis,x_n)$ and let $P(x)$ be a polynomial. 

\smallskip
Concerning the first point, let us assume by contradiction that $ P$ is a solution of the Monge-Ampère equation \eqref{lambda1}. We immediately notice that the left hand side of the equation is a polynomial, while the right hand side is a polynomial only when $P$ is constant. Nevertheless, when $P$ is constant, the left hand side is identically zero, leading to  a contradiction.

\smallskip
Concerning the second point, if we assume by contradiction that $k \log P$ is a solution of the Monge-Ampère equation  \eqref{lambda1} with $\lambda= 0$, then we straightforwardly get that
$$\frac{\det\left[\left(P \frac{\partial^2 P}{\partial{x_\alpha}\partial{x_\beta}}-\frac{\partial P}{\partial{x_\alpha}}\frac{\partial P}{\partial{x_\beta}}\right)x_\alpha+P\frac{\partial P}{\partial{x_\alpha}} \delta_{\alpha\beta}\right]_{1\leq\alpha,\beta\leq n}}{P^{n-1}}=\pm\left(\frac{1}{k} \right)^n P^{-n+1}.$$
We notice that both sides of the  previous equality are polynomials. In particular, the degree of left hand side  cannot be greater that $(n+1)d-n$. Hence, by comparing the degrees of both sides, we get a contradiction. Indeed,
$$(n+1)d-n \geq (n+1)d.$$

\smallskip
Concerning the third point, let $P$ be  a polynomial such that
$$\det\left( \frac{\de^2 P}{\de x_i\de x_j}x_i+\frac{\de P}{\de x_j}\delta_{ij}\right)_{1\leq i,j\leq n}=\pm1.$$
By evaluating the previous equality on the line $x_2=\mathellipsis=x_n=0$, we easily get
\begin{equation}\label{l0P}
\left( x_1 \frac{\de^2 P}{\de^2 x_1}\Big|_{(x_1,0)}+\frac{\de P}{\de x_1}\Big|_{(x_1,0)}\right)\prod_{i=2}^n\frac{\de P}{\de x_i}\Big|_{(x_1,0)}=\pm1.\end{equation}
By taking into account that each term of the previous product is a polynomial, we get that each term needs to be  constant. Therefore, we have in particular that
$$\frac{\de^2 P}{\de x_1\de x_j}\Big|_{(x_1,0)}=0$$
and
$$ x_1 \frac{\de^2 P}{\de^2 x_1}\Big|_{(x_1,0)}+\frac{\de P}{\de x_1}\Big|_{(x_1,0)}=k_1,$$
with $k_1\in\R\setminus\{0\}$. By solving the previous equations and keeping in mind  that $P(x_1,0)$ is a polynomial, we obtain
$$P(x)=k_0+k_1x_1+\sum_{i=2}^n x_i \ Q_i(x_2,\mathellipsis,x_n),$$
where $Q_i$ are polynomials. Moreover, by means of a similar reasoning, namely by restricting the Monge-Ampère equation \eqref{l0P} to the other coordinate axes, we finally get that
$$P(x)=k_0+\sum_{i=1}^n k_i x_i ,$$
with $$\prod_{i=1}^nk_i=\pm 1.$$

\subsection{Proof of Theorem \ref{th.main.2}}\label{sec.th.2}

In order to prove Theorem \ref{th.main.2}, we have to study the case that has remained out the discussion of Section \ref{sec.th.1}, i.e., to study solutions of type  \eqref{diastind_b} of the Monge-Ampère equation \eqref{lambda1} when $\lambda\neq0$. Before doing it, below we give a remark that will be important to give  a further reformulation of Problem \ref{toricprob} in the case under consideration. Then, after proving the technical Lemma \ref{pt}, the second part of Theorem \ref{th.main.2} will be essentially proved in Sections \ref{sec.n1} and \ref{sec.n2}, according to the dimension of the toric \pk manifolds. Section \ref{sec.n3} will complete the proof. 

\begin{rem}\label{l1}
If $\dd_0=\frac{8}{c} \log Q$ is a solution of type  \eqref{diastind_b} of the Monge-Ampère equation \eqref{lambda1}, then the polynomial
\begin{equation}\label{eq.Px} 
Q(x)=1+\sum_{\substack{ 0<|I|\leq d}} a_I x^I,
\end{equation} 
is a solution of the PDE
\begin{equation*} 
\frac{\det\left[\left(Q \frac{\partial^2 Q}{\partial{x_\alpha}\partial{x_\beta}}-\frac{\partial Q}{\partial{x_\alpha}}\frac{\partial Q}{\partial{x_\beta}}\right)x_\alpha+Q\frac{\partial Q}{\partial{x_\alpha}} \delta_{\alpha\beta}\right]_{1\leq\alpha,\beta\leq n}}{Q^{n-1}}=\pm\left(\frac{c}{8}\right)^n Q^{n+1-\frac{4\lambda}{c}}.
\end{equation*}
After the change of coordinates $(x_1,\mathellipsis,x_n)\to \frac{8}{c}(x_1,\mathellipsis,x_n)$, the previous equation reads as 
\begin{equation}\label{MAeq} 
\frac{\det\left[\left(Q \frac{\partial^2 Q}{\partial{x_\alpha}\partial{x_\beta}}-\frac{\partial Q}{\partial{x_\alpha}}\frac{\partial Q}{\partial{x_\beta}}\right)x_\alpha+Q\frac{\partial Q}{\partial{x_\alpha}} \delta_{\alpha\beta}\right]_{1\leq\alpha,\beta\leq n}}{Q^{n-1}}=\pm Q^{n+1-\frac{4\lambda}{c}}.
\end{equation}
Taking into account that the left hand side of \eqref{MAeq} is a polynomial, it straightforwardly follows that 
$\frac{4\lambda}{c}\in\Q$. Moreover, by comparing the degrees of both sides of  \eqref{MAeq}, we get 
$$(n+1)d-n \geq \left(n+1-\frac{4\lambda}{c}\right)d.$$
Hence, $$\frac{4\lambda}{c}\in\Q^+.$$
Let us assume that $\frac{4\lambda }{c}=\frac{s}{q}$, where $\mathrm{gcd}(q,s)=1$. Being  the left hand side of \eqref{MAeq} a polynomial, $Q$ is forced to be the $q$-th power of another polynomial $R$, namely 
$$Q(x_1,\mathellipsis,x_n)=R\left(\frac{x_1}{q},\mathellipsis,\frac{x_n}{q}\right)^q.$$
Since the constant term of $Q$ is equal to $1$, we notice, instead, that the constant term of $R$ can be  equal either to  $1$ or $-1$.
Moreover, one can easily check that $R(x_1,\mathellipsis,x_n)$ is a solution to the equation
\begin{equation}\label{MA**}
\frac{\det\left[\left(R \frac{\partial^2 R}{\partial{x_\alpha}\partial{x_\beta}}-\frac{\partial R}{\partial{x_\alpha}}\frac{\partial R}{\partial{x_\beta}}\right)x_\alpha+R\frac{\partial R}{\partial{x_\alpha}} \delta_{\alpha\beta}\right]_{1\leq\alpha,\beta\leq n}}{R^{n-1}}%=\pm \left(\frac{c}{8q}\right)^n R^{n+1-s}
=\pm R^{n+1-s}.
\end{equation}
Furthermore, if a polynomial $R(x_1,\mathellipsis,x_n)$ is a solution of \eqref{MA**} for some $s\in\N$, then 
$$
P(x_1,\mathellipsis,x_n)=R\left(\frac{x_1}{s},\mathellipsis,\frac{x_n}{s}\right)^{s}
$$ 
is a polynomial solution to the equation 
\begin{equation}\label{MA*}
\frac{\det\left[\left(P \frac{\partial^2 P}{\partial{x_\alpha}\partial{x_\beta}}-\frac{\partial P}{\partial{x_\alpha}}\frac{\partial P}{\partial{x_\beta}}\right)x_\alpha+P\frac{\partial P}{\partial{x_\alpha}} \delta_{\alpha\beta}\right]_{1\leq\alpha,\beta\leq n}}{P^{n-1}}=\pm P^{n}.
\end{equation} 
We notice that the constant term of $P$ is equal to $1$ or $-1$.
\end{rem}
In view of the above remark,  the Reformulation of Problem \ref{toricprob} can be refined, in the case considered in the present section, as follows.

\medskip\noindent
\textbf{\large Refinement of Problem \ref{toricprob}.}\label{Refi}
\emph{Find all polynomials $P^k$ with constant term equal to $1$, with $k\in \Q^+$ and such that $P$ is a polynomial solution,  having constant term  equal to $\pm1$, to the  Monge-Ampère equation \eqref{MA*}}.

\medskip
Indeed, by considering \eqref{eq.x.xi.eta} and \eqref{diastind_b}, we can associate to such $P^k$ the diastasis' function of a \pk metric admitting a \pk  immersion into $(\DP^N,g_{pFS}^N)$, where $N$ need to be at least equal to the number of monomials forming $P^k$ (see \cite{MSpK}):
\begin{equation}\label{eq:appoggio}
\log P(\xi_1\eta_1,\mathellipsis,\xi_n\eta_n)^k. 
\end{equation}
Being $P$ a solution of the  Monge-Ampère equation \eqref{MA*}, 
%the diastasis' function \eqref{eq:appoggio} gives rise to a para-\KE metric when $k=1$. Then, 
it necessarily follows that, for any aforementioned admissible $k$,
any diastasis' function \eqref{eq:appoggio} gives rise to a para-\KE metric which is \pk  immersed into $(\DP^N,g_{pFS}^N)$.

\smallskip\noindent
This can be achieved, in the case when the dimension $n$ is equal either to $1$ (see Section \ref{sec.n1}) or $2$ (see Section \ref{sec.n2}), by using the following lemma, that holds in any dimension. A discussion of the cases with dimension $n\geq 3$ is contained in Section \ref{sec.n3}.
%\textbf{scrivere frase introduttiva\\
%i seguenti lemmi e corollario sono sostanzialemente il Lemma 2.8, Corollario 2.9 e Lemma 2.10 del [NewYork], con qualche leggera differenza, principalmente dovuta al non aver fissato i coefficienti di primo grado uguali a 1}
%
\begin{lem}\label{pt}
The restriction $p(t)$ on a coordinate axis  of a polynomial solution, having constant term  equal to $\pm1$, to the Monge-Amp\`ere equation \eqref{MA*}, reads as:
\begin{equation*}%\label{cauchycond}
  p(t)=\pm\left(1+ \frac{t}{r}\right)^k\,,
\end{equation*}
where $k\in\Z^+$ and $r\in\R$.
\end{lem}
\begin{proof}
Being $p(t)$ be the restriction on the $i$-th coordinate axis (i.e. the line $x_j=0,\text{ for } j\neq i$) of  a  polynomial solution $P$ to the Monge-Amp\`ere equation \eqref{MA*}, we have that
\begin{equation}\label{MArestr}
\left(\left(p(t)\ p''(t)-p'(t)^2\right)t+p(t)\ p'(t)\right) q(t) =\pm p(t)^{n},
\end{equation}
where the $q(t)$ denotes the restriction on the $i$-th coordinate axis of  $\prod_{j\neq i}\frac{\de P}{\de x_j}$.

Let $\{-r_1,\mathellipsis,-r_R\}$ be all the (possibly complex) distinct roots  of  $p$, namely
\begin{equation}\label{p(t)}
p(t)=A\prod_{i=1}^R (t+ r_i)^{k_i},
\end{equation}
with
\begin{equation}\label{distinct}
r_i-r_j\neq 0,\qquad \forall i\neq j.
\end{equation}
Via some straightforward computations, the equation \eqref{MArestr} reads as
\begin{equation}\label{restr}
\left( \sum_{i=1}^R k_i r_i  \prod_{\substack{j=1\\ j\neq i}}^R (t+ r_j)^{2}\right) q(t) =\pm A^{n-2}\prod_{i=1}^R (t+ r_i)^{k_i(n-2)+2}.
\end{equation}
It easily follows that any root of $q(t)$ needs to be  also a root of $p(t)$, namely
$$q(t)=\pm\frac{A^{n-2}}{B} \prod_{i=1}^R (t+ r_i)^{h_i},$$
for some suitable $h_i\in\N$ and $B\in\R\setminus\{0\}$.
Hence, the  equality \eqref{restr} reads as
$$
\sum_{i=1}^R k_i r_i  \prod_{\substack{j=1\\ j\neq i}}^R (t+ r_j)^{2}=B\prod_{i=1}^R (t+ r_i)^{k_i(n-2)-h_i+2}.
$$
If  we assume the existence of an index $i$ such that $k_i(n-2)-h_i+2\neq 0$, then, by evaluating the previous equality at $-r_i$, we  get
$$-k_ir_i \prod_{\substack{j=1\\ j\neq i}}^R ( r_j-r_i)^{2}=0.$$
Nevertheless, it would  contradict \eqref{distinct}. Therefore,
\begin{equation}\label{q(t)}
q(t)=\pm\frac{A^{n-2}}{B} \prod_{i=1}^R (t+ r_i)^{k_i(n-2)+2}
\end{equation}
and
\begin{equation}\label{eq:system.Filippo}
 \sum_{i=1}^R k_i r_i  \prod_{\substack{j=1\\ j\neq i}}^R (t+ r_j)^{2}-B=0.
\end{equation}
Let now consider \eqref{eq:system.Filippo} as a linear system in the variables $k_1,\mathellipsis,k_R$.\\
If $R=1$, such system consists of just one equation, which  has a unique solution: 
$$k_1=\frac{B}{r_1}.$$
If $R\geq 2$, it cannot be compatible for any $t$. Indeed, being the left hand side of \eqref{eq:system.Filippo} a polynomial in $t$ of degree $2R-2$,  in particular its first $R$ coefficients of higher order have to vanish. Therefore, $k_1,\mathellipsis,k_R$ need to satisfy a homogeneous system, whose  determinant of the matrix of coefficients can be easily computed:
$$
R!\prod_{i=1}^R r_i\prod_{1\leq i<j\leq R}(r_i-r_j).
$$
In view of \eqref{distinct}, such determinant is always different from zero. Therefore,  the system admits only the trivial solution, leading to a contradiction, since $k_i$, for any $i$, represents the multiplicity of a root of a polynomial, so it should be positive.

To conclude, since we are assuming that the constant term $\e=Ar_1^{k1}$ of $p$, see \eqref{p(t)}, is equal to $\pm1$,  by taking  into  account \eqref{p(t)} and \eqref{q(t)},  we have  that
$$p(t)=A(t+r_1)^{k_1}=\e\sum_{i=0}^{k_1}\binom{k_1}{i}t^i r_1^{-i}=\e\left( 1+\frac{t}{r_1}\right)^{k_1}$$
and
\begin{equation}\label{q}
q(t)=\pm \e^{n-2} \frac{ r_1}{k_1} \left( 1+\frac{t}{r_1}\right)^{k_1(n-2)+2}.\end{equation}
\end{proof}
%
%
%
%
%\vspace{-0.5cm}
%
\subsubsection{\textbf{Case $n=1$}}\label{sec.n1}

\

\noindent
In the present section, we shall consider the Refinement of Problem \ref{toricprob} at page \pageref{Refi} when $n=1$ and solve it. In view of Lemma \ref{pt},  polynomial solutions to \eqref{MA*}, that we have been studying, need to read as
$$P(x)=\e\left(1+ \frac{x}{r}\right)^k,$$
where $\e=\pm 1$.
Moreover, after some straightforward computations on the Monge-Ampère equation \eqref{MA*} when $n=1$, namely 
\begin{equation*}
\left(P P''-(P')^2\right)x+PP'=\pm P,
\end{equation*}
we get, in particular, that $k=2$ and $r=\pm 2$, i.e.,
$$P(x)=\e\left( 1\pm\frac{x}{2}\right)^2.$$
In view of the Refinement  of Problem \ref{toricprob} at page \pageref{Refi}
we have to take into account the polynomials
$$P(x)=\left( 1\pm\frac{x}{2}\right)^K,\qquad \text{with }K\in\Z^+.$$
By considering \eqref{eq.x.xi.eta} and \eqref{diastind_b}, we can associate to the previous polynomials, the following two families of diastasis functions:
$$
\log\left( 1+\frac{\xi\eta}{2}\right)^K\qquad\text{and}\qquad \log\left( 1-\frac{\xi\eta}{2}\right)^K\,.
$$
We immediately see that the para-holomorphic change of coordinates $(\xi,\eta)\to(-\xi,\eta)$ 
transforms the first diastasis' function into the second, therefore the corresponding \pk metrics, via formula \eqref{eq:metric.from.pot}, are isometric.  In particular, they are isometric to the metric having the first function as a potential:
$$
\frac{2K}{(\xi\eta+2)^2}d\xi d\eta\,.
$$
By considering the further change of coordinates $(\xi,\eta)\to (2\xi,2\eta)$, we see that the previous metric is $$K\, g_{pFS}^1= \frac{2K}{(1+2\xi\eta)^2} d\xi d\eta.$$

\medskip\noindent
Theorem \ref{th.main.2} when $n=1$ is thus proved.

\subsubsection{\textbf{Case $n=2$}}\label{sec.n2}

\

\noindent
In the present section, we are going to find,  when $n=2$, all polynomial solutions of type \eqref{eq.Px} of the Monge-Ampère equation \eqref{MA*}, namely all the polynomials in two variables with constant term equal to 1 solving the following Monge-Ampère equation:
\begin{equation}\label{MAn2}
\frac{\det\left[\left(P \frac{\partial^2 P}{\partial{x_\alpha}\partial{x_\beta}}-\frac{\partial P}{\partial{x_\alpha}}\frac{\partial P}{\partial{x_\beta}}\right)x_\alpha+P\frac{\partial P}{\partial{x_\alpha}} \delta_{\alpha\beta}\right]_{1\leq\alpha,\beta\leq 2}}{P}=\pm P^{2}.
\end{equation}
More precisely, this section is devoted to prove the following proposition.
\begin{prop} \label{newmain}
The only polynomials in two variables of type  \eqref{eq.Px} solving the  Monge-Amp\`ere equation
\eqref{MAn2}
 are
\begin{equation}\label{solution1}
\e\left(1+\frac{1}{r}x_1\pm\frac{r}{9}x_2 \right)^3
\end{equation}
and
\begin{equation}\label{solution2}
\e\left(1+ \frac{1}{r}x_1\right)^2 \left(1\pm\frac{r}{4}x_2\right)^2\,,
\end{equation}
where $r\neq 0$ is an arbitrary real number and $\e=\pm1$.
\end{prop}

\noindent
The proof  follows from the following two lemmas.
\begin{lem}\label{Cauchydata}
An arbitrary polynomial solution, with constant term equal to $\pm1$, of the Monge-Amp\`ere equation  \eqref{MAn2}, satisfies  one and only one of the following initial conditions on the coordinate axis $x_2=0$:
\begin{equation}\label{Cauchy}
\begin{array}{c}
P(x_1,0)=\e\left(1+\frac{x_1}{r}\right)^2,\ \frac{\de P}{\de x_2}(x_1,0)=\pm\frac{r}{2}\left(1+\frac{x_1}{r}\right)^2 \vspace{0,2cm} \\
\vspace{0,2cm}
 \text{ or } \\
P(x_1,0)=\e\left(1+\frac{x_1}{r}\right)^3,\ \frac{\de P}{\de x_2}(x_1,0)=\pm\frac{r}{3}\left(1+\frac{x_1}{r}\right)^2,\\
\end{array}
\end{equation}
where $\e=\pm1$.
\end{lem}

\begin{proof}
Let $P$ be  a solution, whose  constant term is equal to $\pm1$, of  \eqref{MAn2}. By Lemma \ref{pt}, we have that
 \begin{equation}\label{restrizioni}
P(x_1,0)=\e\left(1+\frac{x_1}{r_1} \right)^{k_1},\qquad  P(0, x_2)=\e\left(1+\frac{x_2}{r_2} \right)^{k_2}, \end{equation}
 for suitable $k_1$, $k_2\in\Z^+$ and $r_1,\ r_2\in\R$.  Moreover, by  \eqref{q},
 \begin{equation}\label{derivate}
\frac{\partial P}{\partial x_2}(x_1,0)=\pm\frac{r_1}{k_1} \left( 1+\frac{x_1}{r_1}\right)^{2},\qquad \frac{\partial P}{\partial x_1}(0,x_2)=\pm\frac{r_2}{k_2} \left( 1+\frac{x_2}{r_2}\right)^2.\end{equation}
From the comparison of the derivative w.r.t. $x_1$ of the first equality of \eqref{derivate} and the  derivative w.r.t. $x_2$ of the second equality of \eqref{derivate}, in particular by considering their evaluation at $(0,0)$, we obtain
$$k_1=k_2$$
and 
\begin{equation}\label{derivate2}
\frac{\partial P}{\partial x_2}(x_1,0)=\sigma\frac{r_1}{k_1} \left( 1+\frac{x_1}{r_1}\right)^{2},\qquad \frac{\partial P}{\partial x_1}(0,x_2)=\sigma\frac{r_2}{k_1} \left( 1+\frac{x_2}{r_2}\right)^2\end{equation}
where $\sigma=\pm1$.
Moreover, from the comparison of the second equality of \eqref{derivate2} and the  derivative w.r.t. $x_1$ of the first equality of \eqref{restrizioni}, in particular by considering their evaluation at $(0,0)$, we get 
$$r_2=\e\sigma\frac{k_1^2}{r_1}.$$
Therefore, the polynomial $P$ can be written as:
\begin{multline}\label{soln2}
\e \left(1+\frac{x_1}{r_1} \right)^{k_1}+\e\left(1+\e\sigma\frac{r_1 x_2}{k_1^2} \right)^{k_1}-\e
+\e\frac{k_1 }{r_1} \left( 1+\e\sigma\frac{r_1x_2 }{k_1^2}\right)^2x_1+\sigma \frac{r_1 }{k_1} \left( 1+\frac{x_1}{r_1}\right)^{2}x_2\\
-\e\frac{k_1}{r_1}  x_1-\sigma\frac{r_1}{k_1} x_2-\sigma\frac{2}{k_1}x_1x_2+x_1^2x_2^2\ \eta(x_1,x_2),
\end{multline}
where $\eta$ is a polynomial.

By putting \eqref{soln2} in  \eqref{MAn2}, by differentiating both sides of the equation by $\frac{\partial^2}{\partial x_1\partial x_2}$ and by evaluating at $(0,0)$,
%the second order mixed derivative and evaluating at the origin,
we straightforwardly get the following Diophantine equations:
$$k_1^2-5k_1+6=0$$
and
$$3k_1^2-k_1+6=0.$$
By solving the first equation,  we obtain
$$k_1=2,\quad \text{or}\quad k_1=3,$$
while the second equation does not admit any real solution.
Hence, by considering \eqref{restrizioni} and \eqref{derivate2}, we get our statement.
\end{proof}

Since each solution \eqref{solution1}-\eqref{solution2} satisfies the correspondent initial condition  \eqref{Cauchy}, we   conclude the proof of Proposition \ref{newmain} by proving the following lemma.
%\footnote{Notice that we cannot apply the Cauchy-Kowalevsky Theorem because the Cauchy data \eqref{Cauchy} are characteristic.}:
\begin{lem}
If there exists a polynomial solution of type  \eqref{eq.Px} to \eqref{MAn2}  satisfying one of the initial conditions \eqref{Cauchy}, then it is unique.
\end{lem}
\begin{proof}
Let $F$ be a function whose zero defines the PDE \eqref{MAn2}, i.e.,  
$$F:=\frac{\det\left[\left(P \frac{\partial^2 P}{\partial{x_\alpha}\partial{x_\beta}}-\frac{\partial P}{\partial{x_\alpha}}\frac{\partial P}{\partial{x_\beta}}\right)x_\alpha+P\frac{\partial P}{\partial{x_\alpha}} \delta_{\alpha\beta}\right]_{1\leq\alpha,\beta\leq 2}}{P}\mp P^{2}.$$
Then,
from a straightforward computation, we  get the following formula:
\begin{equation}\label{unique}
\frac{\partial^h F}{\partial x_2^h} (x_1,0)= \left( (h+1)\left( P\frac{\partial^2 P}{\partial x_1^2}x_1-\left(\frac{\partial P}{\partial x_1}\right)^2x_1+P\frac{\partial P}{\partial x_1}\right)\frac{\partial^{h+1} P}{\partial x_2^{h+1}} + T^h\right)(x_1,0)\,,
\end{equation}
where $T^h(x_1,0)$  is a polynomial expression in $x_1$, $P(x_1,0)$ and derivatives of $P$ up to order $h+2$ (computed in $(x_1,0)$), that does not contain $\frac{\partial^{h+1} P}{\partial x_2^{h+1}}(x_1,0)$,  $\frac{\partial^{h+2} P}{\partial x_2^{h+2}}(x_1,0)$ or $\frac{\partial^{h+2} P}{\de x_1 \partial x_2^{h+1}}(x_1,0)$.
If $P$ is a  polynomial solution of type  \eqref{eq.Px}  to \eqref{MAn2} satisfying one of the initial conditions \eqref{Cauchy},  i.e., $P(x_1,0)=\e\left(1+\frac{x_1}{r}\right)^k$ for a suitable integer $k$, hence we have
$$\left( P\frac{\partial^2 P}{\partial x_1^2}x_1-\left(\frac{\partial P}{\partial x_1}\right)^2x_1+P\frac{\partial P}{\partial x_1}\right)(x_1,0)=\frac{k}{r} \left( \frac{x_1}{r}+1\right)^{2k-2}\not\equiv0.$$
By considering formula \eqref{unique} when $h=1$, we realize that  initial conditions \eqref{Cauchy} uniquely determine $\frac{\partial^{2} P}{\partial x_2^{2}} (x_1,0)$, from which one obtains $\frac{\partial^{2+h} P}{\partial x_1^h\partial x_2^{2}} (x_1,0)$ for every $h\in \N$. By iteration, we get  the whole Taylor expansion of $P$ on the line $x_2=0$. Therefore, we get the statement of the lemma.
\end{proof}

By taking into account suitable para-holomorphic change of coordinates and in view of the Refinement  of Problem \ref{toricprob} at page \pageref{Refi}
we have to take into account only the powers of the polynomials of Proposition \ref{newmain} with $\e=1$:
$$\left( 1+\frac{\xi_1\eta_1}{2}\right)^K\left( 1+ \frac{\xi_2\eta_2}{2}\right)^K\,,\qquad \left( 1+ \frac{\xi_1\eta_1}{3} + \frac{\xi_2\eta_2}{3}\right)^K\,,$$
for any $K\in\Z^+$.
By considering \eqref{eq.x.xi.eta} and \eqref{diastind_b}, we can associate to the aforementioned polynomials the following two families of diastasis functions:
$$K\log\left[\left( 1+\frac{\xi_1\eta_1}{2}\right)\left( 1+ \frac{\xi_2\eta_2}{2}\right)\right]\qquad\text{and}\qquad K\log\left( 1+ \frac{\xi_1\eta_1}{3} + \frac{\xi_2\eta_2}{3}\right).$$
Note that the metric we obtain via formula \eqref{eq:metric.from.pot} from the first family is a metric on $\DP^1\times \DP^1$, whereas the second family gives a metric on $\DP^2$ (see the end of Section \ref{sec.pk.diastasis}). 

\medskip\noindent
Theorem \ref{th.main.2} when $n=2$ is thus proved. 

\subsubsection{\textbf{Case $n\geq 3$}}\label{sec.n3} 

\

\noindent
In this section we shall show how to prove the first part of Theorem \ref{th.main.2}. We shall discuss only the case $n=3$ as the multi-dimensional one is a straightforward generalization of it. It will be enough to consider, taking into account \eqref{eq.x.xi.eta}, only the following polynomial solutions to  equation \eqref{MA*}:
\begin{multline}
\left( 1+\frac{\xi_1\eta_1}{2}\right)^2\left( 1+ \frac{\xi_2\eta_2}{2}\right)^2\left( 1+ \frac{\xi_3\eta_3}{2}\right)^2\,,\quad \left( 1+\frac{\xi_1\eta_1}{2}\right)^2 \left( 1+ \frac{\xi_2\eta_2}{3} + \frac{\xi_3\eta_3}{3}\right)^3\,, 
\\
\left( 1+\frac{\xi_1\eta_1}{4}+ \frac{\xi_2\eta_2}{4} + \frac{\xi_3\eta_3}{4}\right)^4\,.
\end{multline}
In view of \eqref{diastind_b} and of  the Refinement  of Problem \ref{toricprob} at page \pageref{Refi}, the above polynomials lead to diastasis functions (in particular, potentials) that, via formula \eqref{eq:metric.from.pot}, give the \pk metrics we are looking for. More precisely, the first polynomial leads to the metric on $\DP^1\times\DP^1\times\DP^1$, the second one to metric on $\DP^1\times\DP^2$ and the third one to metric on $\DP^3$.

\small{}

\end{document}